\documentclass[11pt,a4paper]{article}

\usepackage[utf8x]{inputenc}
\usepackage[english]{babel}
\usepackage{xcolor}
\usepackage{tcolorbox}
\usepackage{url}
\usepackage{extarrows}
\usepackage{placeins}
\usepackage{eurosym}
\usepackage{subfig}
\usepackage{fancyhdr}
\usepackage{hyperref}
\usepackage{indentfirst}
\usepackage{graphicx}
\usepackage{newlfont}
\usepackage{amssymb}
\usepackage{amsmath}
\usepackage{latexsym}
\usepackage{cancel}
\usepackage{amsthm}
\usepackage{listings}
\usepackage{chngcntr}
\usepackage{tikz-cd}
\usepackage{float}
\usepackage{esint}
\newcommand{\myindent}{\hspace*{1em}}
\numberwithin{equation}{section} % Per le equazioni
\newtheorem{Theorem}{Theorem}[section]
\newtheorem{Lemma}[Theorem]{Lemma}
\newtheorem{Proposition}[Theorem]{Proposition}
\newtheorem{Corollary}[Theorem]{Corollary}
\theoremstyle{definition}
\newtheorem{Definition}[Theorem]{Definition}
\theoremstyle{definition}
\newtheorem{Remark}[Theorem]{Remark}
\newcommand{\eps}{\varepsilon}
\newcommand{\curvatura}{\kappa}
\newcommand{\curv}{\curvatura}

\newcommand\ders{\partial_s}
\newcommand\dert{\partial_t}
\newcommand{\domgae}{\textrm{dom}(\gae)}

\newcommand{\ga}{\gamma}
\newcommand{\gae}{\gamma_\eps}

\newcommand{\gaet}{\gamma_\eps(t)}
\newcommand{\gat}{\gamma(t)}

\renewcommand{\k}         {\kappa}

\newcommand{\inidat}{\initialcurve}

\newcommand{\inidatmcf}{\initialcurve}

\newcommand{\initialcurve}{\overline \gamma}

\newcommand{\lgat}{\ell(\gat)}
\newcommand{\lga}{\ell(\gamma)}

\newcommand{\lgaet}{\ell(\gaet)}

\newcommand{\lunghezza}{1}
\newcommand{\menogradFeps}{E_\varepsilon}
\newcommand{\NN}{\mathbb N}

\newcommand\qol{{\mathfrak{q}}}
\newcommand{\R}{\mathbb R}

\newcommand{\tempo}{t^\star}

\newcommand{\Tsingmcf}{T_{\mathrm{sing}}}
\newcommand{\Tmax}{T_{\mathrm{max}}}
\newcommand{\zoga}{\Upsilon}
\newcommand{\zogae}{\Upsilon_\eps} %zo stands for zero one
\newcommand{\nada}[1]{}

\newcommand{\kgaet}{\kappa_{\gae(t)}}
\newcommand{\normgaet}{\nu_{\gae(t)}}
\newcommand{\taugaet}{\tau_{\gae(t)}}
\newcommand{\initialpoint}{P}
\newcommand{\finalpoint}{Q}
\newcommand{\finalpointeps}{Q}
\newcommand{\veltan}{\lambda_\eps}
\title{A fourth-order regularization of the curvature flow of immersed plane curves with Dirichlet boundary conditions}

\author{
		Giovanni Bellettini\footnote{ 
		 Dipartimento di Scienze Matematiche, Informatiche e Fisiche,
via delle Scienze 206, 33100 Udine UD, Italy,
and International Centre for Theoretical Physics ICTP, Mathematics Section, 34151 Trieste, Italy.
		E-mail: giovanni.bellettini@uniud.it}
	\and
	Virginia Lorenzini\footnote{
		Dipartimento di Ingegneria dell'Informazione e Scienze Matematiche, Universit\`a di Siena, 53100 Siena, Italy.
		E-mail: virginia.lorenzini@student.unisi.it
	}
    \and
	Matteo Novaga \footnote{ 
		Dipartimento di Matematica, Universit\`a di Pisa, 56127 Pisa, Italy.
		E-mail: matteo.novaga@unipi.it}
    \and
	Riccardo Scala\footnote{ 
		Dipartimento di Ingegneria dell'Informazione e Scienze Matematiche, Universit\`a di Siena, 53100 Siena, Italy.
		E-mail: riccardo.scala@unisi.it}
}

\begin{document}
\maketitle
\begin{abstract}
\noindent We consider a fourth-order regularization  of the curvature flow for an immersed plane curve with fixed boundary, 
using an elastica-type functional
depending on a small positive parameter $\eps$. 
We show that the approximating flow smoothly 
converges, as $\eps  \to 0^+$, to the curvature flow of the curve with Dirichlet boundary conditions for all times before the first singularity
of the limit flow.
\end{abstract}

\vspace{2mm}

\noindent {\bf AMS (MOS) 2020 Subject Clas\-si\-fi\-ca\-tion:}  53E010, 35K55

%%%%%%%%%%%%%%%%%%%%%%%%%%%%%%%%%%%%%%%%%%%%%%%%%%%%%%%%%%%%%%%%%%%%%%%%%%%%%%%%%%%%
\section{Introduction}

In recent years, considerable attention has been devoted to the study of geometric evolution laws for curves and surfaces, especially those governed by curvature-dependent dynamics. It is well established that singularities may develop in finite time during the evolution, a phenomenon which strongly motivates the study of the flow beyond singularities. In the case of mean curvature motion, various definitions of weak solutions have been introduced, beginning with the foundational work of Brakke \cite{Brakke}.

Following a suggestion of De Giorgi in \cite{DeGiorgi:96}, we introduce and study a fourth-order regularization of mean curvature flow using an elastica-type functional depending on a small positive parameter $\eps$, which could lead to a new definition of generalized solution. More precisely, 
given $\eps\in(0,1]$, we consider a time-dependent family of immersed plane curves 
\(\ga: \bigcup_{t\in[0,T]}(\{t\}\times[0,\ell(\gat)])\rightarrow \mathbb{R}^2\) 
of class \(H^2\cap C^2\), 
with  fixed boundary points 
$\ga(t, 0)=\initialpoint,$ 
$\ga(t, \lga)=\finalpoint$. These curves evolve by 
the gradient flow of the energy
\begin{equation}\label{eq:F_eps}
    F_\varepsilon(\gamma):=
\int_0^{\ell(\gamma)} (1+\varepsilon \kappa_\ga^2) ds,
\end{equation}
where, as usual, $ds=|\partial_x\ga|\ dx$ is the arclength element, $\kappa_{\ga}$ is the curvature of the curve and $\ell(\ga)$ its length.
 In what follows, $\partial_s=|\partial_x \ga|^{-1}\partial_x$ denotes the differentiation with respect to the arclength parameter, while \( \nu_\ga \) and \( \tau_\ga \) denote the oriented normal and tangential vectors of the curve, respectively.
Formally $\gae(t,x)$ satisfies 
\begin{equation}\label{eq:eps_dir_motion_equation}
      (\partial_t\gae)^\perp=\kappa_{\gae(t)}\nu_{\gae(t)}-\eps\Big(2\partial^2_s\kappa_{\gae(t)}+\kappa_{\gae(t)}^3\Big)\nu_{\gae(t)} 
\end{equation}
coupled with the following boundary and initial conditions:
\begin{equation}\label{eq:eps_dir_boundary_conditions}
    \begin{cases}
        \gae(t,0)=\initialpoint & \text{} \\
         \gae(t,\lgaet)=\finalpointeps  & \text{}\\
        \kappa_{\gae(t)}(0)=\kgaet(\lgaet)=0  \\
        \gae(0, \cdot) = \initialcurve(\cdot),
    \end{cases}
\end{equation}
where $\initialcurve$ is an initial regular immersion such that, for some
$\alpha \in (0,1)$,  
\begin{equation}\label{eq:inidatmcf}
\inidatmcf \in C^{4+\alpha}([0,\ell(\inidatmcf)];\R^2) \quad \textrm{with}
\quad 
\inidatmcf(0)=\initialpoint, 
\quad\inidatmcf(\ell(\inidatmcf))=\finalpoint,
\end{equation}
and 
\begin{equation}\label{eq:inidatmcf_curvature}
 \kappa_{\inidat}(0)=0=\kappa_{\inidat}(\ell(\inidat)).
\end{equation}

We point out that, also if we assume that $\overline \gamma$ is an injective curve (with no self-intersections), since the flow does not preserve embeddedness, 
each $\ga_\varepsilon$ is just an immersion (i.e. possibly self-intersecting image). Furthermore, even if $\overline \gamma$ coincides with the graph of some function $f:[a,b]\rightarrow \R$, curves $\ga_\varepsilon$ cannot be assumed to  be graphs as well, since the gradient flow of $F_\eps$ is a fourth
order parabolic system, which in general does not
preserve graphicality.

We can regard
\eqref{eq:eps_dir_motion_equation} as a perturbation of the 
curvature flow, coinciding with it when $\eps=0$. 

In \cite[Theorem 2.1]{Novaga-Okabe:13}~\cite[Theorem 4.15]{Mantegazza-Pluda_Pozzetta:04} it is proven that, 
given $\eps \in (0,1]$ and $\initialcurve$ 
a planar immersed regular
curve satisfying  \eqref{eq:inidatmcf}, \eqref{eq:inidatmcf_curvature}, then the initial boundary value problem  formed by  \eqref{eq:eps_dir_motion_equation}, \eqref{eq:eps_dir_boundary_conditions} 
admits a unique smooth solution $\gae$ defined for all times, that is 
$$
\gae\in C^{\frac{4+\alpha}{4},4+\alpha}(\bigcup_{t \in [0,+\infty)}
\left(\{t\} \times [0,\lgaet]\right); \mathbb{R}^2)\cap 
C^{\infty}(\bigcup_{t \in [\delta,+\infty)}
\left(\{t\} \times [0,\lgaet]\right); \mathbb{R}^2) \qquad\forall \delta>0.
$$
Therefore our 
focus is on the convergence to the curvature flow as $\eps\to 0^+$.

Let us state our main result.
Denote by $\ga$ the immersed curvature flow of
$\inidatmcf$ 
solution to
\begin{equation}\label{eq:initial_bdry_value_prb_intro}
    \begin{cases}
     (\partial_t\ga)^\perp=\kappa_{\gat}\nu_{\gat} \\
     \ga(t, 0)=\initialpoint, \quad \ga(t, \lgat)=\finalpoint \\
     \ga(0, \cdot) = \inidatmcf(\cdot).
    \end{cases}
\end{equation}
It is possible to show 
that \eqref{eq:initial_bdry_value_prb_intro} admits a 
unique solution in a maximal time 
interval $[0,\Tsingmcf)$, with $\Tsingmcf\in(0,+\infty]$. 
Indeed,
 given $\alpha\in(0,1)$ and $\overline{\ga}\in C^{2+\alpha}([0,\ell(\inidat)];\mathbb{R}^2)$ a 
planar immersed curve satisfying $\inidat(0)=\initialpoint,  \ \inidat(\ell(\inidat))=\finalpoint$, 
then there exists $\Tsingmcf>0$ such that 
the initial boundary value problem \eqref{eq:initial_bdry_value_prb_intro} 
admits a unique solution belonging to
$$
C^{\frac{4+2\alpha}{4},2+\alpha}(\bigcup_{t \in [0,\Tsingmcf)}
\left(\{t\} \times [0,\lgat]\right); \mathbb{R}^2)
\cap 
C^{\infty}(\bigcup_{t \in [\delta,\Tsingmcf)}
\left(\{t\} \times [0,\ell(\gat)]\right); \mathbb{R}^2) \qquad\forall \delta \in (0, \Tsingmcf)
$$
 (see  
also \cite{Huisken:98} for a related statement). In addition, one checks
that the condition \eqref{eq:inidatmcf_curvature} is preserved
during the flow.

Given an immersed curve $\ga$ 
(resp. $\ga(t)$,  $\gae$, $\gae(t)$) 
parametrized by arc-length, in what follows we indicate by 
$\zoga$ (resp. $\zoga(t)$,  $\zoga_\eps$, $\zoga_\eps(t)$) 
the constant speed reparametrization of $\ga$
(resp. $\ga(t)$,  $\gae$, $\gae(t)$) on the interval $[0,\lunghezza]$. 

Our main convergence result reads as follows.
\begin{Theorem}[\textbf{Asymptotic convergence}]
\label{teo:asymptotic_convergence}
Let $\inidatmcf$ be a planar immersed regular
curve satisfying  \eqref{eq:inidatmcf}, \eqref{eq:inidatmcf_curvature},
and let 
$\ga$ be the solution to 
\eqref{eq:initial_bdry_value_prb_intro} in 
$[0, \Tsingmcf)$.
For any $\eps\in(0,1]$ 
%let $(\inidat_\eps)$
%be a family of initial immersed curves satisfying properties (i)-(iv) as above and
denote by 
\(\gae\) the solution to \eqref{eq:eps_dir_motion_equation}, \eqref{eq:eps_dir_boundary_conditions}
in $[0,+\infty)$.
%, starting from $\inidat_\eps$.
Then the corresponding
parametrizations $\zogae$ converge in 
$C_{\mathrm{loc}}^\infty((0,\Tsingmcf)\times[0,1])$
to the map $\zoga$, as $\eps \to 0^+$.
Furthermore, if 
$\inidat \in C^\infty([0, \ell(\inidat)])$ and
 all derivatives of $\inidat$ of even order 
at $0$ and $\ell(\inidat)$ vanish, then the convergence
takes place 
in 
$C_{\mathrm{loc}}^\infty([0,\Tsingmcf)\times[0,1])$.
\end{Theorem}

As a consequence, we observe that for $t \in (0, \Tsingmcf)$, 
all space
derivatives of $\gamma(t,\cdot)$ of even order at $0$ and $ \ell(\gat)$ vanish.

The proof
of Theorem \ref{teo:asymptotic_convergence}
 partially follows the lines in \cite{Bellettini_Mantegazza_Novaga:04}, where convergence is established for the curvature flow of regular \textit{closed} curves immersed in \( \mathbb{R}^2 \). The novelty in the present setting lies in 
accounting for boundary terms and for the tangential component of the velocity, that we denote with \( \veltan:=\langle \partial_t \gae, \tau_{\gae}\rangle \).

 The crucial point is to obtain $\eps$-independent integral 
estimates of the 
curvature and all its derivatives.
Indeed, one can see that, by an ODE's argument, 
since all the flows (letting $0<\eps\le 1$ vary) 
start from a common initial smooth curve, 
fixing any $j \in \mathbb{N}$, there exists a common positive interval 
of time such that all the quantities $\|\partial^i_s \k\|_{L^2}$, 
for $i\in\{0,\dots,j\}$ are equibounded. 
This will allow us to get compactness and $C^\infty$ 
convergence to the curvature flow as $\eps\to 0^+$. 
Furthermore, our convergence
 result holds up to  $\Tsingmcf$, and not only
up to $\Tmax$ (defined in Proposition \ref{prop:uniform_estimate_of_k_square}). 
%Our goal would be to provide some limit flow defined \textit{for all times}, thus providing a new weak definition of solution.

We point out that the case of a fourth-order 
regularization of the curvature flow for immersed planar 
curves with different boundary conditions 
is currently under investigation. For examples, Neumann boundary 
conditions, or  
when the boundary points are free to move along two parallel lines, 
and even, more generally, with prescribed trajectories.
However, this analysis appears to be significantly more challenging than the Dirichlet case, because in the latter case, one can exploit the fact that the velocity \( \partial_t \gamma \) vanishes at the boundary. This allows one to discard all boundary terms 
arising from integration by parts. In contrast, under different boundary conditions, we currently do not know how to handle these contributions, which are related with the presence of tangential velocity at the boundary points, a quantity that in general seems
difficult to estimate. 
%%%%%%%%%%%%%%%%%%%%%%%%%%%%%%%%%%%%%%%%%%%%%%%%%%%%%%%%%%%%%%%%%%%%%%%%%%%
\section{Notation and preliminaries}\label{sec:preliminaries}

Given $I:=[0,1]$, 
we consider planar parametrized 
immersed curves $\zoga:I\to\mathbb{R}^2$.  
If $\zoga\in C^k(I, \R^2)$ we say that
$\zoga$ is of class $C^k$. 
A curve of class $C^1$ is regular if $\zoga_x(x)=\frac{d\zoga}{dx}(x)\neq 0$ for every $x\in I$, in which case
its unit tangent vector $\tau= \tau_\gamma= \zoga_x/|\zoga_x|$
is well-defined. 
We indicate by $\nu = \nu_\gamma$ 
the unit normal vector, which is
the counterclockwise rotation $R$ by $\pi/2$ of $\tau$, $\nu = R \tau$. 
The arc-length parameter of curve $\zoga$ (null in zero) is denoted by $s$ and 
is given by
\begin{equation*}
    s:=s(x)=\int^x_0 |\zoga_x(\zeta)|d\zeta, \qquad x \in I.
\end{equation*}
The curve $\Upsilon$, reparametrized by arc-length, 
will be indicated by $\gamma : [0, \ell(\gamma)] \to \R^2$,
with $\tau_\ga=\ga_s$ and 
\begin{equation*}
    \ell(\ga):=\int_I|\zoga_x(x)|dx
\end{equation*}
the length of $\ga$.
If $\ga$ is of class $C^2$, we define the curvature 
vector 
$\ga_{ss} =: 
\kappa_\gamma \nu_\ga$. 
%It turns out that 
%\begin{equation*}
%    \kappa \nu =\frac{\zoga_{xx}|\zoga_x|^2-\zoga_x\langle \zoga_{xx},\zoga_x\rangle}{|\zoga_x|^4}.
%\end{equation*}
We also recall that 
\begin{equation}\label{eq:Serret-Frenet}
     \partial_s\nu_\ga=-\kappa_\ga\tau_\ga.
\end{equation}
In what follows we will consider time-dependent families of curves $(\zoga(t,x))_{t\in[0,T]}$. 
We often write $\gamma(t, \cdot) = \gamma(t)(\cdot)$.
Again, we let $\tau_{\ga(t)}$ be the unit tangent vector to the curve, 
$\nu_{\ga(t)}$ the unit normal vector and $\kappa_{\ga(t)}$ 
its curvature. 

We denote by $\partial_x f$, $\partial_s f$ and $\partial_t f$ the derivatives of a function $f$ along a curve with respect
to the variable $x$, the arc length parameter $s$ on such a curve and the time, respectively. Moreover $\partial^n_x f$, $\partial^n_s f$, $\partial^n_t f$ are the higher order partial derivatives.

When we parametrize $\ga$ by 
constant speed on the interval $[0,1]$ we have $s(t,x) =x\lgat$,
 and $|\partial_x \zoga(t)|=\ell(\gat)$.
 
 In what follows, with a small abuse of notation, we sometimes write 
 $\displaystyle \int_{\gamma} (1+\varepsilon \curv^2) ds$ in place of 
 the right-hand side of \eqref{eq:F_eps}. A similar notation will be adopted for integrals of general quantities on  a curve $\gamma$.
 % 

%%%%%%%%%%%%%%%%%%%%%%%%%%%%%%%%%%%%%%%%%%%%%%%%%%%%%%%%%%%%%%%%%%%%%%%%%%%%
\section{Estimates of geometric quantities}
Let us denote
\begin{equation}\label{eq:energy_eps}
\menogradFeps(t,\cdot) =
\menogradFeps
:=-\kgaet+\eps\Big(2\partial^2_s \kgaet + \kgaet^3\Big)
\end{equation}
the normal velocity of a curve \(\gae\) evolving by the $\eps$-elastic flow. In \eqref{eq:eps_dir_motion_equation} only the normal component of
the velocity is prescribed. This does not mean that the tangential velocity is necessarily zero.
Indeed the motion equation can be written as
\begin{equation}\label{eq:evol_eq_completa}
    \partial_t\gae=-E_\eps\normgaet+\veltan\taugaet\quad \text{in } 
\domgae:= \bigcup_{t \in [0,+\infty)}
\left(\{t\} \times [0,\lgaet]\right),
\end{equation}
where 
$$\veltan=\langle \partial_t\gae,\tau_{\gae}\rangle.
$$
We observe that from the Dirichlet boundary conditions in \eqref{eq:eps_dir_boundary_conditions}, it follows $\frac{d}{dt}\gae(t,s)=0$ for $s\in\{0,\lgaet\}$. 
Therefore,
 from the evolution equation  \eqref{eq:evol_eq_completa} 
we deduce
\begin{align}\label{eq:vel_zero_sinistra}
  E_\eps(t,0)=0, \quad \veltan(t,0)=0 \quad \forall t \in {(0,+\infty)}
\end{align}
and 
\begin{equation*}
\begin{aligned}
 0
 &=\partial_t\gae(t,\lgaet)+\partial_s\gae(t,\lgaet)\dot \ell(\gaet)
 \\
 &= -E_\eps(t,\lgaet)\normgaet(\lgaet)
 +
 [\veltan(t,\lgaet)+\dot \ell(\gaet)]
 \taugaet(\lgaet),
\end{aligned}
\end{equation*}
from which we get 
\begin{equation}\label{eq:vel_norm_zero_destra}
   E_\eps(t,\lgaet)=0 \qquad \forall t \in {(0,+\infty)},
\end{equation}
and 
\begin{equation}\label{eq:vel_tan_destra}
    \veltan(t,\lgaet)=-\dot \ell(\gaet) \qquad \forall t \in {(0,+\infty)}.
\end{equation}
We recall from \cite[Lemma 2.19]{Mantegazza-Pluda_Pozzetta:04} 
the following formulas, which we prove for completeness.                  
\begin{Lemma}[Evolution of geometric quantities]\label{lem:evolution_geometric_quantities}
Let $\eps\in(0,1]$ 
and $\gae:\domgae 
\to\mathbb{R}^2$ be a curve moving by \eqref{eq:evol_eq_completa}. Then 
the following formulas hold:
  \begin{align}
      \partial_t\partial_s-\partial_s\partial_t&=(-\kappa_{\gae}E_\eps-\partial_s\veltan)\partial_s, 
\label{eq:commutation_rule} 
\\
       \partial_t(ds)&=(\kappa_{\gae}E_\eps+\partial_s\veltan) \ ds,
 \label{eq:evolution_ds}
\\
      \partial_t\tau_{\gae}&=(-\partial_sE_\eps+\veltan\kappa_{\gae})
\nu_{\gae}, 
\\
      \partial_t\nu_{\gae}
&=-(-\partial_sE_\eps+\veltan\kappa_{\gae})\tau_{\gae},
  \end{align}
and
  \begin{equation}
      \begin{aligned}
           \partial_t \kappa_{\gae} &= -\partial^2_s \menogradFeps-\kappa_{\gae}^2\menogradFeps+ \veltan\partial_s\kappa_{\gae} 
           \\
            &= \partial^2_s \kgaet+
\kgaet^3-2\varepsilon\partial^4_s \kgaet-6\varepsilon \kgaet(\partial_s
\kgaet)^2-5\varepsilon \kgaet^2\partial^2_s \kgaet-\varepsilon \kgaet^5
\\
&\quad
+ \veltan\partial_s\kappa_{\gae}.  \label{eq:evolution_curvature}
      \end{aligned}
  \end{equation}
\end{Lemma}
\begin{proof}
   We have
   \begin{equation*}
   \begin{aligned}
        \partial_t\partial_s-\partial_s\partial_t&=\frac{\partial_t\partial_x}{|\partial_x\gae|}-\frac{\langle \partial_x\gae,\partial_t\partial_x\gae\rangle\partial_x}{|\partial_x\gae|^3}-\frac{\partial_x\partial_t}{|\partial_x\gae|}
        =-\langle\tau_{\gae},\partial_s\partial_t\gae\rangle\partial_s
        \\
        &= -\langle \tau_{\gae},\partial_s(-E_\eps\nu_{\gae}+
\veltan\tau_{\gae})\rangle\partial_s
=
(-\kappa_{\gae}
E_\eps
-\partial_s\veltan)\partial_s,
   \end{aligned}
   \end{equation*}
where in the last equality we used equation \eqref{eq:Serret-Frenet}. Thus \eqref{eq:commutation_rule} holds, and we can compute 
\begin{align*}
    \partial_t(ds)&=\partial_t|\partial_x\gae|dx
=\frac{\langle \partial_x\gae,\partial_t\partial_x\gae\rangle}{|\partial_x\gae|} dx 
=\langle \tau_{\gae},\partial_s\partial_t\gae\rangle ds
=(\kappa_{\gae} 
E_\eps
+\partial_s\veltan)ds,
\\
     \partial_t\tau_{\gae}&=\partial_t\partial_s\gae=\partial_s\partial_t\gae-(\kappa_{\gae}E_\eps+\partial_s\veltan)\partial_s\gae
       \\
       &=\partial_s(-E_\eps\nu_{\gae}+\veltan\tau_{\gae})-(\kappa_{\gae}E_\eps+\partial_s\veltan)\tau_{\gae}=(-\partial_sE_{\eps}+\kappa_{\gae}\veltan)\nu_{\gae}.
   \end{align*} 
Moreover
       $\partial_t\nu_{\gae}=\partial_t(R\tau_{\gae})=R(\partial_t\tau_{\gae})=-(-\partial_sE_\eps+\kappa_{\gae}\veltan)\tau_{\gae},$
and 
   \begin{align*}
       \partial_t\kappa_{\gae}&=
\partial_t\langle\partial_s\tau_{\gae},\nu_{\gae}\rangle=\langle\partial_t\partial_s\tau_{\gae},\nu_{\gae}\rangle=\langle \partial_s\partial_t\tau_{\gae},\nu_{\gae}\rangle+(-\kappa_{\gae} E_\eps-\partial_s\veltan)\langle\partial_s\tau_{\gae},\nu_{\gae}\rangle 
\\
&= \partial_s\langle\partial_t\tau_{\gae},\nu_{\gae}\rangle-\kappa_{\gae}^2E_\eps-\kappa_{\gae}\partial_s\veltan=\partial_s(-\partial_s E_\eps+\veltan \kappa_{\gae})-\kappa_{\gae}^2E_\eps-\kappa_{\gae}\partial_s\veltan
\\
&
=-\partial^2_sE_\eps+\veltan\partial_s\kappa_{\gae}-\kappa_{\gae}^2E_\eps.
\end{align*}
 By expanding $E_\eps$, the second equality in \eqref{eq:evolution_curvature} follows.
\end{proof}
Fixing $s\in (0,\lgaet)$, and using \eqref{eq:evolution_ds} we have
\begin{align*}
    0
    &=
    \partial_t\int_0^{s} d\sigma
    =\int_0^s \partial_s\veltan(t,\sigma) \ d\sigma+\int_0^s E_\eps \kgaet \ d\sigma 
    \\
    &=
    \veltan(t,s)-\veltan(t,0)+\int_0^s E_\eps \kgaet \ d\sigma 
\end{align*}
from which, recalling \eqref{eq:vel_zero_sinistra}, we get
\begin{equation}\label{eq:velocita_tangenziale}
    \veltan(t,s)=-\int_0^s E_\eps\kgaet \ d\sigma.
\end{equation}
%%%%%%%%%%%%%%%%%%%%%%%%%%%%%%%%%%%%%%%%%%%%%%%%%%%%%%%%%%%%%%%%%%%%%%%%%
\subsection{\texorpdfstring{\(\eps\)-uniform estimates of length, 
energy and normal velocity}{epsilon-uniform estimates of length, energy and normal velocity}}
Obviously 
\begin{equation}\label{eq:lgaet_geq_1}
\lgaet \geq |\initialpoint-\finalpointeps| \qquad \forall 
t \in [0,+\infty), \  \forall \eps \in (0,1].
\end{equation}
Moreover
there exists a constant $C=C(\inidat)>0$ such that 
\begin{equation}\label{eq:upper_bound_on_the_length}
\sup_{t \in [0,+\infty)}
\sup_{\eps \in (0,1]}    
\lgaet\leq C\ \qquad \forall\eps\in(0,1].
\end{equation}
Indeed, 
$$
\ell(\gaet) \leq F_\eps(\gaet) \leq F_\eps(\inidat) \leq \int_0^{\ell(\inidat)}
\Big(1+ \curv_{\inidat}^2\Big)~ds \leq C,
$$
as a consequence 
of the fact that the PDE in \eqref{eq:eps_dir_motion_equation} is the 
gradient flow of $F_\eps$ and $\eps \in (0,1]$.
%and due to assumption (ii) on $\overline \ga_\eps$ in the Introduction.

\begin{Lemma}[Energy dissipation equality]\label{lem:dissipation_eq}
For any $\eps \in (0,1]$ and any 
$t\in(0,+\infty)$ we have 
 \begin{equation}\label{eq:energy_monotonicity_bis}
\begin{aligned}
    \frac{d}{dt} F_\eps(\gae(t))=&
-\int_0^{\lgaet} E_\eps^2(t,s) ds
\\
=&-\frac{1}{2}\int_0^{\lgaet}|(\partial_t\gae(t))^\perp|^2ds 
-\frac{1}{2}\int_0^{\lgaet} E_\eps^2(t,s) ds.
\end{aligned}
\end{equation}
\end{Lemma}
\begin{proof} 
Using \eqref{eq:evolution_ds}, \eqref{eq:evolution_curvature} and \eqref{eq:vel_tan_destra}, we get
$$
\begin{aligned}
\frac{d}{dt} F_\eps(\gaet)) 
=&\frac{d}{dt}\int_0^{\lgaet} \left(1 +\eps \kgaet^2\right) ds
\\
=& \int_0^{\lgaet} \Big( 2\eps\kgaet\partial_t \kgaet+ (1+\eps\kgaet^2)(
\kgaet E_\eps+\partial_s\veltan) \Big)\ ds
\\
&\quad
+\Big(1 +\eps \kgaet^2(\lgaet)\Big)\dot\ell(\gaet)
\\
    =&\int_0^{\lgaet} \Big(2\eps\kgaet (-\partial^2_s \menogradFeps-\kappa_{\gae}^2\menogradFeps+\veltan\partial_s\kgaet)+
    (1+\eps\kgaet^2)(
\kgaet E_\eps+\partial_s\veltan) \Big) ds
\\
&\quad
-\veltan(t,\lgaet)\Big(1 +\eps \kgaet^2(\lgaet)\Big)
\\=&
\int_0^{\lgaet} \Big(-2\eps\kgaet
\partial^2_sE_\eps
-E_\eps(\eps\kgaet^3-\kgaet)+\partial_s(\veltan(1+\eps\kgaet^2))\Big) ds
\\
&\quad
-\veltan(t,\lgaet)\Big(1 +\eps \kgaet^2(\lgaet)\Big).
\end{aligned}
$$
Integrating twice by parts the term  $ -2\eps\kgaet\partial^2_sE_\eps$
and observing that 
$$
-2\eps \partial^2_s \kgaet E_\eps - E_\eps 
(\eps \kgaet^3 - \kgaet) = - E_\eps(2 \eps \partial^2_{s} \kgaet+ 
\eps \kgaet^3 - \kgaet) = - E_\eps^2,
$$

 we obtain
\begin{equation*}
    \frac{d}{dt} F_\eps(\gaet)=-\int_0^{\lgaet} E_\eps^2 \ ds 
+ 2\eps \Big[- \kgaet\partial_s 
E_\eps+\partial_s\kgaet E_\eps\Big]^{\ell(\gae(t))}_0-\veltan(t,0)\Big(1+\eps\kgaet^2(0)\Big),
\end{equation*}
where we recall that $\gae\in C^{\infty}(\domgae; \mathbb{R}^2)$ 
for any $\delta>0$.

It remains to show that 
\begin{equation}\label{eq:bordo_energia}
2\eps\Big[- \kgaet\partial_s 
E_\eps(t,\cdot)+\partial_s\kgaet E_\eps(t,\cdot)\Big]^{\ell(\gae(t))}_0-\veltan(t,0)\Big(1+\eps\kgaet^2(0)\Big)=0.
\end{equation}
From \eqref{eq:eps_dir_boundary_conditions} we have 
$$ \kgaet(s)=0 \quad \text{ for }  s\in\{0,\lgaet\}$$
which, together with \eqref{eq:vel_zero_sinistra} and \eqref{eq:vel_norm_zero_destra}, 
yields \eqref{eq:bordo_energia}.
\end{proof}
\begin{Remark}
    In the proof of Lemma \ref{lem:dissipation_eq} we heavily 
use the Dirichlet boundary conditions. Considering 
other boundary conditions would present the challenge of handling the boundary terms in equation~\eqref{eq:bordo_energia}.
\end{Remark}

\begin{Corollary}[$L^2$-estimate of the normal velocity]
There exists a constant $C = C_{\inidat}>0$ such that
for any $T\in (0,+\infty)$ 
 \begin{equation}\label{eq:estimate_of_gamma_t}
      \sup_{\eps \in (0,1]} \int_0^T \Vert (\partial_t \gae(t))^\perp\Vert^2_{L^2([0,\lgaet];\R^2)} \ dt \leq C.
    \end{equation}
\end{Corollary}
\begin{proof}
From Lemma \ref{lem:dissipation_eq}, integrating over time, we have 
    \begin{equation*}
    \int_0^T \frac{d}{dt} F_\eps(\gaet) \ dt=- \int_0^T\int_0^{\lgaet}(E_\eps)^2 ds
\end{equation*}
from which
\begin{equation*}
    F_\eps(\gae(T,\cdot))+ 
\int_0^T\int_0^{\lgaet} E_\eps^2 ds \ dt=  F_\eps(\gae(0,\cdot))=F_\eps(\initialcurve)\leq \int_0^{\ell(\inidat)}
\Big(1+ \curv_{\inidat}^2\Big)~ds\leq C,
\end{equation*}
where we have used the initial condition in \eqref{eq:eps_dir_boundary_conditions}.
As a consequence
\begin{align*}
  \int_0^T\int_0^{\lgaet} E_\eps^2 ds \ dt= \int_0^T \Vert (\partial_t 
\gae(t))^\perp\Vert^2_{L^2([0,\lgaet];\R^2)} \ dt \leq C.
\end{align*}
\end{proof}

%%%%%%%%%%%%%%%%%%%%%%%%%%%%%%%%%%%%%%%%%%%%%%%%%%%%%%%%%%%%%%%%%%%%%%%%%%%%%%%%
\subsection{\texorpdfstring{\(\eps\)-uniform estimate of \(\Vert \kappa_{\gae}\Vert_2\)}{epsilon-uniform estimate of norm}}
Let 
\(\gae\) be the solution to \eqref{eq:eps_dir_motion_equation},
\eqref{eq:eps_dir_boundary_conditions}.
One of the crucial ingredients to prove Theorem \ref{teo:asymptotic_convergence}
 is the following

\begin{Proposition}\label{prop:uniform_estimate_of_k_square}
There exists $\Tmax >0$ such that, for any $T \in (0, \Tmax)$ 
\begin{equation}\label{eq:uniform_estimate_of_k_square}
\sup_{\eps\in(0,1]}
\sup_{t\in[0,T]}
\int_0^{\ell(\gaet)} \kgaet^2~ds < +\infty.
\end{equation}
\end{Proposition}

In order to prove this proposition \ref{prop:uniform_estimate_of_k_square}
we need some preliminaries. 
Although the statement of the next
lemma coincides with that of \cite[Lemma 3.9]{Bellettini_Mantegazza_Novaga:04}, it must be reproven in our setting, as it requires taking into account both the boundary contributions and the tangential component of the velocity.

\begin{Lemma}\label{lem:dt_kappa_2}
 For any $t \in (0,+\infty)$ 
    we have
    \begin{equation}\label{eq:dt_kappa_2}
\begin{aligned}
       \frac{d}{dt}\int_0^{\lgaet} \kappa_{\gaet}
^2 ds=
&\int_0^{\lgaet}
\Big(-2(\partial_s \kappa_{\gaet})^2
+\kappa_{\gaet}^4\Big) ds
\\
& +\varepsilon\int_0^{\lgaet}\Big(
-4(\partial^2_s(\kappa_{\gaet}))^2- \kappa_{\gaet}^6-4 
\kappa_{\gaet}^3\partial^2_s \kappa_{\gaet}\Big)ds.
\end{aligned}
    \end{equation}
\end{Lemma}
\begin{proof}
Using equality \eqref{eq:evolution_ds}, we get
\begin{align*}
 &\frac{d}{dt}\int_0^{\lgaet} \kgaet^2 \ ds 
\\
&\quad=
\int_0^{\lgaet} 2\kgaet\partial_t \kgaet \ ds+
\int_0^{\lgaet}\Big(-\kgaet^4+2\varepsilon \kgaet^3\partial^2_s\kgaet+\varepsilon \kgaet^6+\kgaet^2\partial_s\veltan\Big)ds
\\
&\qquad
+\kgaet^2(\lgaet)\dot \ell(\gaet).
\end{align*}
From \eqref{eq:eps_dir_boundary_conditions} we have $\kgaet(\lgaet)=0$, thus the term 
 $\kgaet^2(\lgaet)\dot \ell(\gaet)=0$.
Using \eqref{eq:evolution_curvature}, and coupling together the terms $2\kgaet\veltan\partial_s\kgaet$ and $\kgaet^2\partial_s\veltan$, we obtain 
\begin{equation}\label{eq:again}
\begin{aligned}
&\frac{d}{dt}\int_0^{\lgaet} \kgaet^2 \ ds   
\\
&\quad=
\int_0^{\lgaet} \Big(2\kgaet\partial^2_s\kgaet+\kgaet^4\Big)ds
\\
&\qquad+\varepsilon\int_0^{\lgaet}\Big(-4 \kgaet\partial^4_s\kgaet-12 \kgaet^2(\partial_s \kgaet)^2-8 \kgaet^3\partial^2_s \kgaet- \kgaet^6\Big) ds
\\
&\qquad +
\int_0^{\lgaet} \partial_s(\veltan\kgaet^2) \ ds.
\end{aligned}
\end{equation}
Again from \eqref{eq:eps_dir_boundary_conditions} 
we deduce that $\displaystyle\int_0^{\lgaet} \partial_s(\veltan\kgaet^2) \ ds=0$.

We now argue as in \cite{Bellettini_Mantegazza_Novaga:04} 
with the difference to be pointed out here  that we need
to take into account the boundary terms.
The first and third terms on the right hand side of \eqref{eq:again}
 depend on higher 
derivatives and, moreover, they lack a sign. 
Therefore, we integrate by parts (with respect to $s$) once for the term without \(\varepsilon\)  and twice for the term 
with \(\varepsilon\) obtaining:
\begin{equation}\label{eq:partial_t_int_kae_squared}
    \begin{aligned}
&\frac{d}{dt}\int_0^{\lgaet} \kgaet^2 \ ds 
\\
&\quad= \int_0^{\lgaet} \Big(-2(\partial_s\kgaet)^2
+\kgaet^4\Big)ds \\
&
\qquad
+\varepsilon\int_0^{\lgaet}\Big(-4(\partial^2_s\kgaet)^2
-12 \kgaet^2(\partial_s \kgaet)^2-8 \kgaet^3\partial^2_s \kgaet 
- 
\kgaet^6
\Big) ds 
\\
& \qquad+ 
\Big[2\kgaet\partial_s\kgaet-4\varepsilon \kgaet\partial^3_s \kgaet+
4\varepsilon 
\partial_s\kgaet\partial^2_s \kgaet
\Big]^{\lgaet}_{0} 
\\
&\quad= \int_0^{\lgaet} \Big(-2(\partial_s\kgaet)^2+\kgaet^4\Big)ds+\varepsilon\int_0^{\lgaet}\Big(-4(\partial^2_s\kgaet)^2
-4 \kgaet^3\partial^2_s \kgaet
- 
\kgaet^6
\Big) ds 
\\         
&\qquad+ 
\Big[2\kgaet\partial_s\kgaet
-
4\varepsilon \kgaet\partial^3_s \kgaet
+4\varepsilon \partial_s\kgaet\partial^2_s \kgaet
\Big]^{\lgaet}_{0}  
    \end{aligned}
\end{equation}
where in the last equality,
taking also into account the third formula in \eqref{eq:eps_dir_boundary_conditions},
 we used 
that \(\displaystyle -3\int_0^{\lgaet} \kgaet^2(\partial_s \kgaet)^2 ds=\int_0^{\lgaet} \kgaet^3\partial^2_s \kgaet ds.\)
   
It remains to show that the contribution of the boundary term is zero.
From \eqref{eq:eps_dir_boundary_conditions} we have $\kgaet(s)=0, s\in\lbrace 0, \lgaet\rbrace$, thus we only need to show that 
    
\begin{equation}\label{eq:vanishing_of_third_boundary_term}
\left[
%4\varepsilon(
\partial_s \kgaet\partial^2_s \kgaet
%)
\right]^{\lgaet}_{0} 
=0 \qquad \forall t \in {(0,+\infty)}.
\end{equation}
This follows from the fact that, for $s\in\lbrace 0, \lgaet\rbrace$,
\begin{align*}
    0=E_\eps(t,s)=-\kgaet(s)+\eps\Big(2\partial^2_s \kgaet(s) + \kgaet^3(s)\Big)=2\eps\partial^2_s \kgaet(s)
\end{align*}
 where we have used \eqref{eq:vel_zero_sinistra}, \eqref{eq:vel_norm_zero_destra}, \eqref{eq:energy_eps} and $\kgaet(s)=0, s\in\lbrace 0, \lgaet\rbrace$.
\end{proof}

We now recall the following version of the Gagliardo–Nirenberg inequality
which follows from \cite[Theorem 1]{Nirenberg:96}
  and a scaling argument.
\begin{Theorem}[\textbf{Gagliardo-Nirenberg interpolation inequalities}]
\label{teo:Gagliardo_Nirenberg_interpolation_inequalities}
    Let \(\eta\) be a smooth curve in \(\mathbb{R}^2\) with 
length $L \in (0,+\infty)$ and let \(u\) be a smooth function defined on \(\eta\). Then for every \(j\geq 1\), \(p\in [2,+\infty]\) and \(n\in \lbrace 0,\dots,j-1\rbrace\) we have the estimates
    \begin{equation*}
        ||\partial^n_s u||_{L^p}\leq \tilde{C}_{n,j,p}||\partial^j_s u||^\sigma_{L^2}||u||^{1-\sigma}_{L^2}+\frac{B_{n,j,p}}{L^{j\sigma}}||u||_{L^2}
    \end{equation*}
    where
    \begin{equation*}
        \sigma=\frac{n+1/2-1/p}{j}
    \end{equation*}
    and the constants \(\tilde{C}_{n,j,p}\) and \(B_{n,j,p}\) are independent of \(\eta\). In particular, if \(p=+\infty\),
    \begin{equation*}
        ||\partial^n_s u||_{L^\infty}\leq\tilde{C}_{n,j}||\partial^j_s u||^\sigma_{L^2}||u||^{1-\sigma}_{L^2}+\frac{B_{n,j}}{L^{j\sigma}}||u||_{L^2} \hspace{0.2cm}\text{ with }\hspace{0.2cm} \sigma=\frac{n+1/2}{j}.
    \end{equation*}
\end{Theorem}

\begin{Remark}\label{rem:particular_case_of_Gagliardo_Nirenberg}
If
 \(n=0, j=2, p=6\) we get \(\sigma=1/6\) and
\begin{equation*}
    ||u||_{L^6}\leq C||\partial^2_s u||^{\frac{1}{6}}_{L^2}||u||^{\frac{5}{6}}_{L^2}+\frac{C}{L^{\frac{1}{3}}}||u||_{L^2},
\end{equation*}
for some \(C>0\), hence, by means of Young inequality  \(|xy|\leq\frac{1}{a}|x|^a+\frac{1}{b}|y|^b,\) \(1/a+1/b=1\), choosing \(a=b=2\), \(x=\sqrt{2}||\partial^2_s u||^{1/2}_{L^2}\) and \(y=\frac{1}{\sqrt{2}}||u||^{5/2}_{L^2}\), we obtain
\begin{equation}\label{eq:GN_for_u^6}
    \int_\gamma u^6 ds\leq \int_\gamma (\partial^2_s u)^2 ds+C\bigg(\int_\gamma u^2 ds\bigg)^5+\frac{C}{L^2}\bigg(\int_\gamma u^2 ds\bigg)^3.
\end{equation}
If 
 \(n=0, j=1, p=4\) we get \(\sigma=1/4\) and
\begin{equation*}
    ||u||_{L^4}\leq C||\partial_s u||^{\frac{1}{4}}_{L^2}||u||^{\frac{3}{4}}_{L^2}+\frac{C}{L^{\frac{1}{4}}}||u||_{L^2},
\end{equation*}
hence, reasoning as before,
\begin{equation}\label{eq:GN_for_u^4}
    \int_\gamma u^4 ds\leq \int_\gamma (\partial_s u)^2 ds+C\bigg(\int_\gamma u^2 ds\bigg)^3+\frac{C}{L}\bigg(\int_\gamma u^2 ds\bigg)^2.
\end{equation}
\end{Remark}

The next result is proven 
in \cite[Proposition 3.10]{Bellettini_Mantegazza_Novaga:04}, and holds
also for a curve with boundary. The idea of the proof is 
to use Lemma \ref{lem:dt_kappa_2}, adding a 
suitable positive quantity in order to eliminate 
terms whose sign is not known, and then estimating the 
resulting expressions using \eqref{eq:GN_for_u^6} and \eqref{eq:GN_for_u^4}.
\begin{Lemma}\label{lem:partial_t_int_k_squared_leq_C_int_k_squared}
There exists a positive constant $C>0$ independent of $\eps\in(0,1]$ such that
the following estimate holds:
\begin{equation}
\begin{aligned}\label{eq:k^2}
& 
\frac{d}{dt} \int_0^{\lgaet}
\kappa_{\gaet}^2~ds 
\\
\leq& C\bigg (\int_0^{\lgaet} 
\kappa_{\gaet}^2 ds\bigg)^5+C\bigg(\int_0^{\lgaet}\kappa_{\gaet}^2 
ds\bigg)^3
+C\bigg(\int_0^{\lgaet} \kappa_{\gaet}^2 ds\bigg)^2.
\end{aligned}
\end{equation}
\end{Lemma}
\bigskip
{\it Proof of Proposition \ref{prop:uniform_estimate_of_k_square}}
    The statement follows by Lemma \ref{lem:partial_t_int_k_squared_leq_C_int_k_squared}, estimate \eqref{eq:lgaet_geq_1} and by the Gronwall's lemma.
\qed

\bigskip

\nada{
\begin{Remark}\label{rem:no_possible_estimate_on_derivative}
If we believe in the convergence of \eqref{eq:eps_dir_motion_equation} to the mean curvature flow and of
\eqref{eq:eps_dir_boundary_conditions}
to the Dirichlet boundary conditions, 
we cannot hope for an $\eps-L^2$-uniform estimate of $\partial_s
\kgaet$,
\begin{equation}\label{eq:uniform_estimate_of_k_s_square}
\sup_{\eps \in (0,1)} \int_0^{\lgaet} (\partial_s \kappa_{\gaet})^2~ds < +\infty.
\end{equation}
Indeed, $\sup_\eps \int_\gamma (\partial_s \kappa_{\gaet})^2~ds < +
\infty$, coupled with \eqref{eq:uniform_estimate_of_k_square}
would yield a uniform $L^\infty$-estimate of $\kgaet$,
which would imply that the condition $\kgaet=0$ on $\{0, \lgaet\}$
must be preserved in the limit. 
 The loss of 
 estimate \eqref{eq:uniform_estimate_of_k_s_square}
is in striking contrast
with the case of closed curves \cite{Bellettini_Mantegazza_Novaga:04}, where it is shown that there is a uniform
bound for the $L^2$-norm of the derivative of $\kgaet$ of any order.
{}From the analytical point of view, it turns out that when trying to
get \eqref{eq:uniform_estimate_of_k_s_square} one get several
boundary terms that cannot be estimated. This leads
to the idea of finding a uniform estimate of the $L^2$-norm of $\kgaet^2$.
\end{Remark}
}
%%%%%%%%%%%%%%%%%%%%%%%%%%%%%%%%%%%%%%%%%%%%%%%%%%%%%%%%%%%%%%%%%%%%%%%%%%
\subsection{\texorpdfstring{\(\eps\)-uniform estimates of \(\Vert \partial^j_s\kappa_{\gae}\Vert_2\)}{epsilon-uniform estimates of higher derivatives}}

We deal now with the higher derivatives of the curvature, 
obtaining this crucial result:

\begin{Proposition}\label{prop:uniform_estimate_of_higher_derivatives_square}
There exists $\Tmax >0$ such that, for any $T \in (0, \Tmax)$ and $j\in\mathbb{N}$
\begin{equation}\label{eq:uniform_estimate_of_higher_derivatives_square}
\sup_{\eps\in(0,1]}
\sup_{t\in[0,T]}
\int_0^{\ell(\gaet)} (\partial^j_s\kgaet)^2~ds < +\infty.
\end{equation}
\end{Proposition}
In order to prove Proposition \ref{prop:uniform_estimate_of_higher_derivatives_square} we need some preliminaries.

\begin{Definition}\label{def:q}
For $l\in\mathbb{N}$, we denote by $\qol^r(\partial^l_s\kappa)$ a polynomial in $\kappa, \dots, \partial_s^l\kappa$ with constant coefficients in $\mathbb{R}$ such that every monomial it contains is of the form
\begin{equation*}
    \prod_{i=1}^N \partial^{j_i}_s\kappa \qquad \text{with } 0\leq j_i\leq l, \ N\geq 1 \text{ and }  r=\sum_{i=1}^N (j_i+1).
\end{equation*}
\end{Definition}
For the rest of this section all polynomials in the curvature $\kgaet$ and its derivatives are completely contracted, that is they belong to the family $\qol^r(\partial^l_s\kgaet)$ as defined above.

The following formula is proved in \cite[Lemma 2.19]{Mantegazza-Pluda_Pozzetta:04}:
\begin{Lemma}
Let $\eps\in(0,1]$ and $\gae:\domgae \to\mathbb{R}^2$ be a 
curve evolving by \eqref{eq:evol_eq_completa}. Then,  for any $j\in\mathbb{N}$, we have
\begin{equation}\label{eq:evolution_j_derivata_curvatura}
    \begin{aligned}
     \partial_t\partial^j_s\kgaet
     &=\partial^{j+2}_s\kgaet+\qol^{j+3}(\partial^j_s\kgaet)
     -2\eps\partial^{j+4}_s\kgaet-5\eps\kgaet^2\partial^{j+2}_s\kgaet
     \\
     &\quad
     +\eps\qol^{j+5}(\partial^{j+1}_s\kgaet)+\veltan\partial^{j+1}_s\kgaet.
    % \\
    % &=\qol^{j+3}(\partial^{j+2}_s\kgaet)-2\eps\partial_s^{j+4}\kgaet+ \eps \qol^{j+5}(\partial^{j+2}_s\kgaet)+\veltan\partial^{j+1}_s\kgaet   
    \end{aligned}
\end{equation}
\end{Lemma}

\myindent From the conditions (see \eqref{eq:eps_dir_boundary_conditions}, \eqref{eq:vel_zero_sinistra}, \eqref{eq:vel_norm_zero_destra})
\begin{equation}\label{eq:kappa_zero_al_bordo}
\kgaet=0 \qquad {\mathrm at}~ \{0, \lgaet\},
\end{equation}
\begin{equation*}\label{eq:velocita_zero_al_bordo}
\langle\dert \gaet,\normgaet\rangle= E_\eps =0 \qquad {\mathrm at}~ \{0, \lgaet\},
\end{equation*}
 and using the 
expression of $E_\eps$ in \eqref{eq:energy_eps}
it follows 
\begin{equation}\label{eq:der2_kappa_zero_al_bordo}
\partial_s^2 \kgaet =0 \qquad {\mathrm at}~
\{0, \lgaet\}, \quad t\in {(0,+\infty)}.
\end{equation}
Also 
from the condition $\partial_t \kgaet(0)=0$, and from 
\begin{equation}\label{eq:velocita_zero_al_bordo_bis}
   \langle\dert \gae(t,0),\taugaet(0)\rangle= \veltan(t,0)=0 
\end{equation}
 (see \eqref{eq:vel_zero_sinistra}), using \eqref{eq:kappa_zero_al_bordo},
\eqref{eq:der2_kappa_zero_al_bordo}
 and the evolution equation 
\eqref{eq:evolution_curvature} it follows 
\begin{equation}\label{eq:der4_kappa_sinistra}
\partial_s^4 \kgaet(0) =0,  \quad t\in (0,+\infty).
\end{equation}
Looking now at $s=\lgaet$, from the condition
\begin{align*}
        0
        &=\frac{d}{dt}\kgaet(\lgaet)=\partial_t\kgaet(\lgaet)+\partial_s\kgaet(\lgaet)\dot \ell(\gaet),
    \end{align*}
and from the evolution equation \eqref{eq:evolution_curvature}, \eqref{eq:kappa_zero_al_bordo}, \eqref{eq:der2_kappa_zero_al_bordo} and 
\begin{equation}\label{eq:vel_tan_destra_bis}
    \langle \partial_t\gae(t,\lgaet),\tau_{\gaet}(\lgaet)\rangle=\veltan(t,\lgaet)=-\dot\ell(\gaet)
\end{equation}
(see \eqref{eq:vel_tan_destra}),
we obtain
\begin{equation}\label{eq:der4_kappa_destra}
    \begin{aligned}
        0= & -2\eps\partial^4_s\kgaet(\lgaet)+\veltan(t,\lgaet)\partial_s\kgaet(\lgaet)+\partial_s\kgaet(\lgaet)\dot \ell(\gaet)
        \\
        &=-2\eps\partial^4_s\kgaet(\lgaet)+\veltan(t,\lgaet)\partial_s\kgaet(\lgaet)-\veltan(t,\lgaet)\partial_s\kgaet(\lgaet)
        \\
        &=-2\eps\partial^4_s\kgaet(\lgaet).
    \end{aligned}
\end{equation}
Combining \eqref{eq:der4_kappa_sinistra} and \eqref{eq:der4_kappa_destra} we conclude 
\begin{equation}\label{eq:der4_kappa_zero_al_bordo}
\partial_s^4 \kgaet =0,  \qquad {\mathrm at}~
\{0, \lgaet\}, \quad t\in (0,+\infty).
\end{equation}
For the following lemma see also \cite[Lemma 4.8]{Mantegazza-Pluda_Pozzetta:04} and \cite[Lemma 2.5]{Novaga-Okabe:13}.
\begin{Lemma}\label{lem:derivate_j_k_pari}
    For any $t\in(0,+\infty)$
\begin{equation}\label{eq:first_remark}
\partial^j_s \kgaet =0 \qquad {\mathrm at}~ 
\{0,\lgaet\} \quad \forall j \mathrm{ ~even.}
\end{equation}
\end{Lemma}
\begin{proof}We write $j=2n$, and we argue by induction on $n$. 
    The first step of induction $n=1$ is given by \eqref{eq:der2_kappa_zero_al_bordo}.
    Let $n\in\mathbb{N}$, $n\geq 2 $ and suppose that $\partial^{2m}_s\kgaet(0)=\partial^{2m}_s\kgaet(\lgaet)=0$ holds for any natural number $m\leq n$. 
    We have to show that 
    \begin{align}\label{eq:passo_induttivo}
        \partial^{2(n+1)}_s\kgaet=0 \qquad {\textrm at}~
\{0, \lgaet\}.
    \end{align}
    Using  \eqref{eq:evolution_j_derivata_curvatura} with $j=2(n-1)$, we have, at $s=0$,
\begin{align*}
    0=\partial_t\partial^{2(n-1)}_s\kgaet
     &=\partial^{2n}_s\kgaet
     +\qol^{2n+1}(\partial^{2n-2}_s\kgaet)
     -2\eps\partial^{2(n+1)}_s\kgaet
     -5\eps\kgaet^2\partial^{2n}_s\kgaet
     \\
     &\quad
     +\eps\qol^{2n+3}(\partial^{2n-1}_s\kgaet)
     +\veltan\partial^{2n-1}_s\kgaet
     \\
     &=\qol^{2n+1}(\partial^{2n-2}_s\kgaet)
     -2\eps\partial^{2(n+1)}_s\kgaet
     +\eps\qol^{2n+3}(\partial^{2n-1}_s\kgaet),
\end{align*}
where we use \eqref{eq:kappa_zero_al_bordo} and \eqref{eq:velocita_zero_al_bordo_bis}, 
and the induction hypothesis which ensures 
$\partial^{2n}_s\kgaet(0)=0$.

 Using again the induction hypothesis and \eqref{eq:evolution_j_derivata_curvatura} with $j=2(n-1)$, we have, at $s=\lgaet$,
\begin{align*}
    0&=\frac{d}{dt}\partial^{2(n-1)}_s\kgaet=\partial_t\partial^{2(n-1)}_s\kgaet
     +\partial^{2(n-1)}_s\kgaet \dot\ell(\gaet)\\
     &=\partial^{2n}_s\kgaet
     +\qol^{2n+1}(\partial^{2n-2}_s\kgaet)
     -2\eps\partial^{2(n+1)}_s\kgaet
     -5\eps\kgaet^2\partial^{2n}_s\kgaet
     \\
     &\quad
     +\eps\qol^{2n+3}(\partial^{2n-1}_s\kgaet)
     +\veltan\partial^{2n-1}_s\kgaet
     -\veltan\partial^{2n-1}_s\kgaet
     \\
     &=\qol^{2n+1}(\partial^{2n-2}_s\kgaet)
     -2\eps\partial^{2(n+1)}_s\kgaet
     +\eps\qol^{2n+3}(\partial^{2n-1}_s\kgaet),
\end{align*}
where we use also \eqref{eq:vel_tan_destra_bis}. 
Now we observe that 
$$\qol^{2n+1}(\partial^{2n-2}_s\kgaet)=0, \quad \qol^{2n+3}(\partial^{2n-1}_s\kgaet)=0 \quad {\textrm at}~
\{0, \lgaet\}.$$ 
Actually, we shall prove more, namely that each monomial
of $\qol^{2n+1}(\partial^{2n-2}_s\kgaet)$ and of 
\\
$ \qol^{2n+3}(\partial^{2n-1}_s\kgaet)$ vanishes at $\{0, \lgaet\}$.
Indeed, recalling Definition \ref{def:q}, we have that $\qol^{2n+1}(\partial^{2n-2}_s\kgaet)$ is a polynomial in $\kgaet, \dots,\partial^{2n-2}_s\kgaet$ such that every of its monomials is of the form 
\begin{equation*}
    \prod_{i=1}^N \partial^{j_i}_s\kgaet \qquad \text{with } 0\leq j_i\leq 2n-2, \text{ for some }\ N\geq 1 
\end{equation*}
and 
\begin{equation}\label{eq:2n+1}
     2n+1=\sum_{i=1}^N (j_i+1).
\end{equation}
Now, we observe that it is not possible that all indices $j_i$, for $i\in\{1,\dots, N\}$ are odd, as a direct consequence of \eqref{eq:2n+1}. Therefore, there exists at least one index $i\in\{1,\dots,N\}$ such that $j_i$ is even. Since $j_i \leq 2n-2$, the inductive hypothesis implies that $\partial_s^{j_i} \kgaet = 0$ at $\{0,\lgaet\}$. Consequently, the entire monomial vanishes.

Similarly, monomials of $\qol^{2n+3}(\partial^{2n-1}_s\kgaet)$ are of the form 
\begin{equation*}
    \prod_{i=1}^N \partial^{j_i}_s\kgaet \qquad \text{with } 0\leq j_i\leq 2n-1, \text{ for some } N\geq 1 
\end{equation*}
and 
\begin{equation}\label{eq:2n+3}
     2n+3=\sum_{i=1}^N (j_i+1).
\end{equation}
As in the previous case, it is not possible that all indices  $j_i$, for $i\in\{1,\dots, N\}$ are odd,
since this would contradict \eqref{eq:2n+3}. Therefore, there exists at least one index $i\in\{1,\dots,N\}$ such that $j_i$ is even. Since $j_i \leq 2n-1$, the inductive hypothesis implies that $\partial_s^{j_i} \kgaet = 0$ at $\{0,\lgaet\}$ and the entire monomial vanishes. This completes the proof of \eqref{eq:passo_induttivo} and \eqref{eq:first_remark} follows. 
\end{proof}

Using Lemma \ref{lem:derivate_j_k_pari} we prove that 
\cite[Lemma 3.12]{Bellettini_Mantegazza_Novaga:04} is still valid,
since we can show that the boundary terms at any order vanish:

\begin{Proposition}
%Suppose that all derivatives of even order of $\kappa_{\gaet}$
%vanish on $\{t\} \times \{0,\ell(\gaet)\}$. Then 
For any $t \in (0,+\infty)$  and any $j\in\NN$ we have
\begin{align*}
\frac{d}{dt} \int_{\gaet} (\ders^j \kgaet)^2 \,ds
= & -2\int_{\gaet} (\ders^{j+1} \kgaet)^2\,ds
-4\eps \int_{\gaet} (\ders^{j+2} \kgaet)^2 \,ds
\\
& +
 \int_{\gaet} \qol^{2j+4}(\ders^j\kgaet)\,ds
+ \eps \int_{\gaet} \qol^{2j+6}(\ders^{j+1}\kgaet)\,ds\,.
\end{align*}
\end{Proposition}

\begin{proof}
Using \eqref{eq:evolution_ds}
and  \eqref{eq:evolution_j_derivata_curvatura} we deduce,
also coupling together the term containing $\lambda_\eps$
in $\partial_t \partial_s^j \kappa_{\gaet}$
and the term $\partial_s^j \kappa_{\gaet} \partial_s \lambda_\eps$, 
\begin{align*}
    \frac{d}{dt}\int_0^{\lgaet} (\partial^j_s\kgaet)^2 \,ds&=
    2\int_0^{\lgaet} \partial^j_s\kgaet\partial_t\partial^j_s\kgaet \,ds 
+\int_0^{\lgaet} (\partial^j_s\kgaet)^2 
(\kappa_{\gae}E_\eps+\partial_s\veltan) \ ds 
    \\
    &\quad + \Big(\partial^j_s\kgaet(\lgaet)\Big)^2\dot \ell(\gaet)
    \\
    &=2\int_0^{\lgaet} \partial^j_s \kgaet(\partial^{j+2}_s\kgaet+\qol^{j+3}(\partial^j_s\kgaet)) \,ds
    \\
    &\quad 
    + \eps\int_0^{\lgaet} 2\partial^j_s\kgaet
\Big(-2\partial^{j+4}_s
\kgaet-5\kgaet^2\partial^{j+2}_s\kgaet +\qol^{j+5}(\partial^{j+1}_s\kgaet)\Big) 
\,ds 
    \\
    &\quad
    -\int_0^{\lgaet} (\partial^j_s\kgaet)^2 \kgaet
\Big(\kgaet-2\eps\partial^2_s\kgaet-\eps\kgaet^3\Big) \,ds
    \\
    &\quad 
    +\int_0^{\lgaet}\partial_s(\veltan (\partial^{j}_s\kgaet)^2) \,ds
    + \Big(\partial^j_s\kgaet(\lgaet)\Big)^2\dot \ell(\gaet).
\end{align*}
Notice that \eqref{eq:velocita_zero_al_bordo_bis} and \eqref{eq:vel_tan_destra_bis} 
imply that the last line of the above expression vanishes, i.e.
\begin{align*}
    \veltan(t,\lgaet)(\partial^{j}_s\kgaet(\lgaet))^2-\veltan(t,0)
(\partial^{j}_s\kgaet(0))^2
    + \Big(\partial^j_s\kgaet(\lgaet)\Big)^2\dot \ell(\gaet)=0.
\end{align*}
Integrating by parts we deduce 
\begin{align*}
    \frac{d}{dt} \int_{\gaet} (\ders^j \kgaet)^2 \,ds
= & -2\int_{\gaet} (\ders^{j+1} \kgaet)^2\,ds
-4\eps \int_{\gaet} (\ders^{j+2} \kgaet)^2 \,ds
\\
& +
 \int_{\gaet} \qol^{2j+4}(\ders^j\kgaet)\,ds
+ \eps \int_{\gaet} \qol^{2j+6}(\ders^{j+1}\kgaet)\,ds\,
\\
& +2 
\Big[
\partial_s^j \kappa_{\gaet}
\partial_s^{j+1} \kappa_{\gaet}\Big]^{\lgaet}_0
-4 \eps
\Big[
\partial_s^j \kappa_{\gaet}
\partial_s^{j+3} \kappa_{\gaet}\Big]^{\lgaet}_0
\\
&
+4 \eps
\Big[
\partial_s^{j+1} \kappa_{\gaet}
\partial_s^{j+2} \kappa_{\gaet}\Big]^{\lgaet}_0.
\end{align*}
 Now  Lemma \ref{lem:derivate_j_k_pari} 
implies 
$$
\partial_s^j \kappa_{\gaet}
\partial_s^{j+1} \kappa_{\gaet}
= 
\partial_s^j \kappa_{\gaet}
\partial_s^{j+3} \kappa_{\gaet}
=
\partial_s^{j+1} \kappa_{\gaet}
\partial_s^{j+2} \kappa_{\gaet}=0 \qquad {\textrm at}~ \{0, \lgaet\},
$$
which ensures that the boundary term vanish. The thesis follows.
\end{proof}
The next result is proven in \cite[Proposition 3.13]{Bellettini_Mantegazza_Novaga:04}, and is also valid in our setting.
\begin{Proposition}\label{prop:higher_derivatives}
    For any $j\in\mathbb{N}$ we have the $\eps$- independent estimate, for $\eps\in(0,1)$,
    \begin{equation}\label{eq:higher_derivatives}
        \frac{d}{dt}\int_0^{\lgaet} (\partial^j_s\kgaet)^2 \ ds\leq C\bigg(\int_0^{\lgaet} \kgaet^2 \ ds\bigg)^{2j+3}+C\bigg(\int_0^{\lgaet} \kgaet^2 \ ds\bigg)^{2j+5}+C
    \end{equation}
    where the constant $C$ depends only on $1/\lgaet$.
\end{Proposition}
By means of Propositions \ref{prop:uniform_estimate_of_k_square} and \ref{prop:higher_derivatives} we have then the following result.
\begin{Theorem}\label{teo316}
    For any $j\in\mathbb{N}$ there exists a smooth function $Z^j:\mathbb{R}\to (0,+\infty)$ such that
    \begin{equation}\label{eq:higher_der}
        \frac{d}{dt}\int_0^{\lgaet} 
(\partial^j_s\kgaet)^2 \ ds\leq Z^j\bigg(\int_0^{\lgaet} \kgaet^2 \ ds\bigg)
    \end{equation}
    for every $\eps\in(0,1)$.
\end{Theorem}
\begin{proof}
    The statement follows by Propositions \ref{prop:uniform_estimate_of_k_square} and \ref{prop:higher_derivatives}, since by \eqref{eq:lgaet_geq_1} the quantity $1/\lgaet$ is controlled.
\end{proof}
\bigskip
We are now ready to prove Proposition \ref{prop:uniform_estimate_of_higher_derivatives_square}.

%\textcolor{red}{controllare un attimo le notazioni nella prova qui sotto}
\textit{Proof of Proposition \ref{prop:uniform_estimate_of_higher_derivatives_square}.}
From
\begin{equation*}
\int_0^{\lgaet} (\partial^j_s\kappa_{\gaet})^2~ds
=\int_0^t \frac{d}{dt}\int_0^{\ell (\gae(\tau))} 
(\partial^j_s\kappa_{\gae(\tau)})^2 \ ds \ d\tau 
+ \int_0^{\ell(\initialcurve)} 
(\partial_s^j \kappa_{\initialcurve})^2 \ ds,
\end{equation*}
using \eqref{eq:higher_der} and the smoothness of $\initialcurve$, we get
\begin{align*}
\int_0^{\lgaet} (\partial^j_s\kappa_{\gaet})^2~ds
\leq 
\int_0^t Z^j\bigg(\int_0^{\ell(\gae(\tau))} \kappa^2_{\gae(\tau)} \ ds\bigg)+ C,
\end{align*}
where $C$ is a positive constant independent of $\eps$ and depending only on $\inidat$.
Now we conclude by applying \eqref{eq:uniform_estimate_of_k_square}.
\qed
\begin{Remark}
As a direct consequence of Proposition \ref{prop:uniform_estimate_of_higher_derivatives_square}, we deduce that for every $T\in(0,\Tmax)$ there exists a constant \( C > 0 \), depending on $T$ and  \(\inidat \), such that
\begin{equation}\label{eq:vel_perp}
  \sup_{\eps\in(0,1]}  |E_\eps(t,x)| 
\leq C \quad \forall (t,x) \in [0,T] \times [0,1].
\end{equation}

Indeed, Proposition \ref{prop:uniform_estimate_of_higher_derivatives_square} ensures that all spatial derivatives of the curvature are uniformly bounded in \( L^2 \) in $[0,T]$.
%that is, given $n\in\mathbb{N}$,
%\[
%\Vert \partial^j_s \kappa_{\gae}(t,\cdot) \Vert_{L^2([0,1])} < C \quad \forall j\leq n, \quad \forall t \in (0, \Tsingeps).
%\]
Therefore, for every \( n \in \mathbb{N} \),
\[
\Vert \kappa_{\gae}(t,\cdot) \Vert_{W^{n,2}([0,\lgaet])} < C(n) \qquad \forall t\in[0,T].
\]
Hence, by Sobolev embedding, there exists a constant \( \hat{C}(n) > 0 \) such that
\[
\Vert \partial^n_s \kappa_{\gae}(t,\cdot) \Vert_{L^\infty([0,\lgaet])} \leq \hat{C}(n).
\]
Finally, recalling the expression of \( E_\eps \) (see \eqref{eq:energy_eps}), which depends on \( \kgaet \), \( \partial_s^2 \kgaet \), and \( \kgaet^3 \), estimate \eqref{eq:vel_perp} directly follows.
\end{Remark}
\begin{Remark}
A further consequence of \eqref{eq:upper_bound_on_the_length} and \eqref{eq:uniform_estimate_of_k_square} is that there exists a constant $C>0$ 
independent of $\eps$ such that for any $t \in [0,\Tmax)$ 
and any $\eps \in (0,1]$,
\begin{equation*}\label{eq:derivative_of_the_length}
\begin{aligned}
    \frac{d}{dt}\lgaet\leq &
C\bigg(\int_0^{\lgaet} \kappa_{\gaet}^2 ds\bigg)^3+\frac{C}{\lgaet}\bigg(\int_0^{\lgaet} 
\kappa_{\gaet}^2 ds\bigg)^2\leq C.
\end{aligned}
\end{equation*}
Indeed, using equations \eqref{eq:evolution_ds}, \eqref{eq:velocita_zero_al_bordo_bis}, \eqref{eq:vel_tan_destra_bis} and integrating by parts
we get 
\begin{align*}
        \frac{d}{dt}\ell(\gaet)&=\frac{d}{dt}\int_0^{\lgaet} 1 \ ds
=  
\int_0^{\lgaet}\Big(
E_\eps \kgaet 
+\partial_s\veltan\Big) \,ds \,+ \dot\ell(\gaet) \\
&=\int_0^{\lgaet}
\Big(-\kgaet^2+2\varepsilon \kgaet\partial^2_s \kgaet
+\varepsilon \kgaet^4 \Big)\ ds 
\\
& \ \ \ \ +\veltan(t,\lgaet)-\veltan(t,0)+ \dot\ell(\gaet)
\\
&=\int_0^{\lgaet}
\Big(-\kgaet^2- 2\varepsilon (\partial_s \kgaet)^2+\varepsilon \kgaet^4
\Big) \ ds 
+
2\eps
\Big[
\kgaet\partial_s\kgaet\Big]^{\lgaet}_0.
\end{align*}
Using the boundary conditions on \(\kgaet\) in
\eqref{eq:eps_dir_boundary_conditions} and \eqref{eq:GN_for_u^4},
we find 
\begin{align*}
\frac{d}{dt}\ell(\gaet)&\leq
\int_0^{\lgaet}\Big(
 -\kgaet^2-2\varepsilon (\partial_s \kgaet)^2 +\varepsilon(\partial_s \kgaet)^2 
\Big)~ds \\
&\quad  
+\eps C\bigg(\int_0^{\lgaet} \kgaet^2 ds\bigg)^3+
\eps
\frac{C}{\lgaet} \bigg(\int_0^{\lgaet} \kgaet^2 ds\bigg)^2
\\
&\leq 
 C\bigg(\int_0^{\lgaet} \kgaet^2 ds\bigg)^3+
\frac{C}{\lgaet} \bigg(\int_0^{\lgaet} \kgaet^2 ds\bigg)^2,
\end{align*}
   since
$\varepsilon \in (0,1]$. We conclude using \eqref{eq:lgaet_geq_1} and \eqref{eq:uniform_estimate_of_k_square}.
\qed
\end{Remark}
%%%%%%%%%%%%%%%%%%%%%%%%%%%%%%%%%%%%%%%%%%%%%%%%%%%%%%%%%%%%%%%%%%%%%%%%%%%%%
\subsection{\texorpdfstring{\(\eps\)-uniform estimate 
of the tangential velocity}{epsilon-uniform estimate of tangential velocity}}
From the expression \eqref{eq:velocita_tangenziale} we obtain
\begin{equation*}
\begin{aligned}
    |\veltan(t,s)| 
    &= \left| -\int_0^s E_\eps \kgaet \, d\sigma \right| 
    \leq 
    \Vert E_\eps \Vert_{L^2([0,\lgaet])} \, \Vert \kgaet \Vert_{L^2([0,\lgaet])} \\
    &\leq 
    \lgaet \, \Vert E_\eps \Vert_{L^\infty([0,\lgaet])} \, \Vert \kgaet \Vert_{L^2([0,\lgaet])}.
\end{aligned}
\end{equation*}
Thus, using   
 \eqref{eq:upper_bound_on_the_length}, \eqref{eq:uniform_estimate_of_k_square} and \eqref{eq:vel_perp},  we deduce that for every $T\in(0,\Tmax)$ there exists a constant \( C > 0 \), 
depending on $T$ and  \(\inidat \), such that
\begin{equation}\label{eq:vel_tan}
   \sup_{\eps\in(0,1]} |\veltan(t,x)| \leq C \quad \forall (t,x) \in [0,T] \times [0,1].
\end{equation}
\section{Proof of Theorem \ref{teo:asymptotic_convergence}}
Let $g_0>0$ and
$Z:(0, +\infty) \to (0,+\infty)$ be a smooth function, for which we assume $$\liminf_{x\rightarrow \infty}Z(x)>0,$$ so that 
the maximal 
forward solution to the Cauchy problem
\begin{equation}\label{cauchy}
	\begin{cases}
		g'(t)=Z(g(t)),
\\
		g(0)=g_0, 
	\end{cases}
\end{equation}
is defined on $[0,a)$, 
 for some   
$a \in (0,+\infty)\cup\{+\infty\}$.
Write  $g([0,a)) = [g_0,+\infty)$.
Function $g$ is continuous and invertible since $g'>0$,  
therefore its inverse
$g^{-1}:[g_0, +\infty) \to [0,a)$ is continuous and strictly increasing.

\begin{Lemma}[\textbf{Doubling time}]\label{lem:ODE}
There exists a 
continuous function $\Theta:(0,+\infty)\to(0,+\infty)$
such that for every $T 
\in [0,a)$ it holds $T+\Theta(g(T))<a$, and for all $t\in[T, T+\Theta(g(T))]$ we have $g(t)\leq 2g(T)$.
\end{Lemma}
\begin{proof}
We start to define
\begin{equation*}
\Theta(s):=g^{-1}(2s)-g^{-1}(s), \qquad s \in 
[g_0, +\infty),
\end{equation*}
which a positive continuous function.
Fix $T\in [0,a)$ and set $s:=g(T)$.  
By definition of $\Theta$ we have
$T+\Theta(s)=g^{-1}(2s)$.
Hence for any $t\in[T,\,T+\Theta(s)]$ we have
$t\in[g^{-1}(s),\,g^{-1}(2s)]$,
and by monotonicity of $g$,
$$
g(t)\in[g(g^{-1}(s)),\,g(g^{-1}(2s))]=[s,\,2s].
$$
Thus $g(t)\le 2s=2g(T)$. This argument shows the existence
of a function $\Theta$ defined on $[g_0, +\infty)$. 
Now, for any $n\in \NN$, the previous construction 
with $1/n$ in place of $g_0$ provides a function $\Theta_n$ defined on $[1/n, +\infty)$, 
and by the uniqueness of the solutions of \eqref{cauchy}
and the explicit expression of
$g^{-1}$,
if $m < n$, we have $\Theta_m = \Theta_n$ on 
$[1/m, +\infty)$. 
This proves the existence of a function $\Theta$ defined on 
the whole of $(0,+\infty)$ and satisfying the required properties.
\end{proof}

We will now show Theorem \ref{teo:asymptotic_convergence}. Recall that $\Tmax$
is defined in Proposition \ref{prop:uniform_estimate_of_k_square}
and $\Tmax = \sup \{ T > 0 : \zogae \to \zoga \text{ in } C^\infty((0,T] \times [0,1])\}$
 and that 
$\Tsingmcf$  is the first singularity time for $\zoga$.

\medskip
\textit{Proof of Theorem \ref{teo:asymptotic_convergence}}.
From \eqref{eq:evol_eq_completa}, 
\eqref{eq:vel_perp} and \eqref{eq:vel_tan}, we know that for every \( T\in(0,\Tmax) \)  there exists a constant \( C > 0 \), independent of \( \eps \), such that
\begin{equation*}
	\bigg|\frac{\partial
\zogae(t,x)
}{\partial t}\bigg|\leq|E_\eps(t,x)|+|\veltan(t,x)| \leq C 
	\quad \forall(t,x)\in [0,T]
	\times [0,1], \quad  \forall \eps\in(0,1].
\end{equation*}
Moreover, 
the constant-speed parametrization $\zogae(t,\cdot)$
over $[0,1]$ of $\gae(t,\cdot)$ does not degenerate, 
due to the lower bound on the length \eqref{eq:lgaet_geq_1} 
(and to the upper bound \eqref{eq:upper_bound_on_the_length}).

Hence, by the Ascoli-Arzel\`a's theorem in time-space, up to a 
not relabelled subsequence, 
the immersions $\zogae$ uniformly converge, as $\eps\to 0^+$, to some 
continuous map $\hat\zoga:[0,T]\times[0,1]\to \mathbb{R}^2$.

From \eqref{eq:uniform_estimate_of_higher_derivatives_square} we have $\eps$-uniform bounds on 
$\kgaet$ and all its derivatives with respect to $s$, and thus, 
from equation \eqref{eq:eps_dir_motion_equation}, also 
on all its derivatives with respect to time 
of any order. Thus 
the convergence of $\zogae$ is in $C_{\mathrm{loc}}^\infty((0,T] \times [0,1])$ 
and 
 passing to the limit in \eqref{eq:eps_dir_motion_equation}, 
\eqref{eq:eps_dir_boundary_conditions}, 
it follows  that $\hat\zoga$ satisfies \eqref{eq:initial_bdry_value_prb_intro}. 
Since the solution of \eqref{eq:initial_bdry_value_prb_intro} is unique, all the sequence $(\zogae)$ converges to $\hat{\zoga}$ which hence coincides with $\zoga$. 
Notice that $\hat \zoga$ is smooth in $(0,T]$, 
so that 
$$
T_{\max} \leq
\Tsingmcf.
$$
We aim to show 
$\Tmax=\Tsingmcf$, namely that
$\Tmax\geq \Tsingmcf$.
Assume by contradiction that 
$$
\Tmax<\Tsingmcf.
$$
Then, by the smoothness 
of $\zoga$ on $(0, \Tsingmcf)$, we have 
\begin{equation}\label{eq:T_max}
\zoga(t,\cdot) \to \zoga(\Tmax,\cdot) \text{ in } C^\infty([0,1]) \text{ as } t \to \Tmax^-.
\end{equation} 

Set 
$$
g_0:=\int_0^{\ell(\inidat)}\kappa_{\inidat}^2 \ ds,
$$
and consider the solution $g$ of the Cauchy problem \eqref{cauchy} 
with 
$Z$ given by the right hand side of \eqref{eq:k^2}, namely 
$$
Z(p) := C p^5 + C p^3 + C p^2, \qquad p >0.  
$$
Let $\Theta: (0,+\infty) \to (0,+\infty)$ be given by 
Lemma~\ref{lem:ODE}.
By \eqref{eq:T_max} and the continuity of $\Theta$ 
we get
\begin{align*}
	 \lim_{t\to \Tmax^-} \Theta\left( \int_0^{\ell(\ga(t))} \kappa_{\ga(t)}^2 \, ds \right) = \Theta\left( \int_0^{\ell(\ga(\Tmax))} \kappa_{\ga(\Tmax)}^2 \, ds \right)=:\theta>0.
\end{align*}
Therefore we can pick 
$$
\tempo\in[T_{\max}-\theta/4,\,T_{\max})
$$
such that
\begin{equation*}
\Theta\!\Big(\int_0^{\ell(\gamma(\tempo))}\kappa_{\gamma(\tempo)}^2\,ds\Big)\ge \tfrac{2}{3}\theta.
\end{equation*}
Our aim is to show
the convergence 
\begin{equation}\label{eq:our_aim_is}
\zogae(t, \cdot) \to \zoga(t, \cdot) \quad \textrm{in} \quad   C^\infty([0,1]) 
\quad \textrm{for}\quad  t \in  [\tempo, \tempo + \theta/2].
\end{equation}
 As $\tempo
+\theta/2\geq \Tmax-\theta/4+\theta/2>\Tmax$, we will get the contradiction.

By the convergence $\zogae(\tempo,\cdot)\to\zoga(\tempo,\cdot)$ in $C^\infty([0,1])$ as $\eps\to0^+$ we have
\begin{equation*}
\lim_{\eps\to0^+}\int_0^{\ell(\gamma_\eps(\tempo))}\kappa_{\gamma_\eps(\tempo)}^2\,ds
= \int_0^{\ell(\gamma(\tempo))}\kappa_{\gamma(\tempo)}^2\,ds.
\end{equation*}
Hence there exists $\overline\eps>0$ such that 
\begin{equation}\label{eq:Theta_two}
\Theta\!\Big(\int_0^{\ell(\gamma_\eps(\tempo))}\kappa_{\gamma_\eps(\tempo)}^2\,ds\Big)\ge \tfrac{\theta}{2} \qquad 
\forall \eps\in(0,\overline\eps].
\end{equation}
Let us now define
\begin{equation}\label{eq:h,tildaT}
h_\infty := \sup_{\eps \in (0,\overline\eps]} \int_0^{\ell(\gae(\tempo))} \kappa_{\gae(\tempo)}^2 \, ds,
\end{equation}
which is finite. Continuity of $\Theta$ and 
\eqref{eq:Theta_two} imply
%We have\footnote{Indeed,
%select an infinitesimal 
%sequence $(\varepsilon_n) \subset (0,\overline\varepsilon]$) such that
%$\int_0^{\ell(\gamma_{\varepsilon_n}(\tempo))}\kappa_{\gamma_{\varepsilon_n}(\tempo)}^2\,ds \to h$.
%Since by 
%\eqref{eq:Theta_two}
%for each $n$ we have $\Theta\!\Big(\int_0^{\ell(\gamma_{\varepsilon_n}(\tempo))}
%\kappa_{\gamma_{\varepsilon_n}(\tempo)}^2\,ds\Big)\ge \theta/2$ and $\Theta$ is continuous, passing to the limit gives
%\eqref{eq:Theta_h_ge_theta2}.
%} 
\begin{equation}\label
{eq:Theta_h_ge_theta2}\Theta(h_\infty) \ge \tfrac{\theta}{2}.
\end{equation}
Set
$$
T_\infty := g^{-1}(h_\infty) >0.
$$
By \eqref{eq:Theta_h_ge_theta2} 
we know $\Theta(g(T_\infty))=\Theta(h_\infty)\ge \theta/2$,
so that 
\begin{equation}\label{eq:useful_inclusion}
[\tempo,\;\tempo+\Theta(g(T_\infty)))
\supseteq
[\tempo,\,\tempo+\theta/2].
\end{equation}
Recall (see \eqref{eq:higher_der} with $j=0$ and $Z = Z^0$) that 
\begin{equation}\label{eq:diff_ineq}
\partial_t\Big(\int_0^{\ell(\gamma_\varepsilon(t))}\kappa_{\gamma_\varepsilon(t)}^2\,ds\Big)
\le Z\!\Big(\int_0^{\ell(\gamma_\varepsilon(t))}\kappa_{\gamma_\varepsilon(t)}^2\,ds\Big)
\qquad \forall \eps \in (0,1], \ 
\ \  \forall t>0.
\end{equation}

By the definition of $h_\infty$ in \eqref{eq:h,tildaT} we have
\begin{equation}\label{eq:initial_le_h}
\int_0^{\ell(\gamma_\varepsilon(\tempo))}\kappa_{\gamma_\varepsilon(\tempo)}^2\,ds \le h_\infty
\qquad
\forall \varepsilon\in(0,\overline\varepsilon].
\end{equation}
If $\tau\in[\tempo,\;\tempo+\Theta(h_\infty))$ then
\begin{equation*}
\tau-(\tempo-T_\infty)\in[T_\infty,\;T_\infty+\Theta(h_\infty))   \cap [0,a).
\end{equation*}
Define 
\begin{equation*}
\sigma(\tau) := g\big(\tau-(\tempo-T_\infty)\big) \qquad \forall
\tau \in 
[\tempo,\;\tempo+\Theta(h_\infty)).
\end{equation*}
Then
\begin{equation}\label{eq:sigma_ode}
\begin{cases}
\sigma'(\tau)=g'(\tau-(\tempo-T_\infty))=Z\big(g(\tau-(\tempo-
T_\infty))\big)=Z(\sigma(\tau)), &{}
\\
\sigma(\tempo)=g(T_\infty)=h_\infty. &{}
\end{cases}
\end{equation}

By \eqref{eq:initial_le_h} we have 
\begin{equation}\label{eq:compa}
\int_0^{\ell(\gamma_\varepsilon(\tempo))}
\kappa_{\gamma_\varepsilon(\tempo)}^2\,ds\le\sigma(\tempo)
\qquad \forall
\varepsilon\in(0,\overline\varepsilon]. 
\end{equation}
Therefore, applying the comparison principle for ODEs 
in $[\tempo, \tempo + \Theta(h_\infty))$ between:
\begin{itemize}
\item[-]  the function $t \in 
(0,+\infty)\mapsto \int_0^{\ell(\gamma_\varepsilon(t))}
\kappa_{\gamma_\varepsilon(t)}^2\,ds$, 
which satisfies \eqref{eq:diff_ineq} (hence is a subsolution of \eqref{eq:sigma_ode}), 
and
\item[-] $\sigma$, solution of \eqref{eq:sigma_ode} with the same 
initial condition at $\tau=\tempo$, 
\end{itemize}
we obtain: for every $\varepsilon\in(0,\overline\varepsilon]$ and 
every $\tau\in[\tempo,\;\tempo+\Theta(h_\infty))$,
\begin{equation}\label{eq:comparison_final}
\int_0^{\ell(\gamma_\varepsilon(\tau))}\kappa_{\gamma_\varepsilon(\tau)}^2\,ds
\le \sigma(\tau)=g\big(\tau-(\tempo-T_\infty)\big).
\end{equation}

Combining \eqref{eq:comparison_final} 
with the bound 
\begin{equation*}
g(t)\le 2g(T_\infty)
\qquad\text{for every } t\in[T_\infty,\;T_\infty+\Theta(h_\infty)) \cap [0,a)
\end{equation*}
given by Lemma \ref{lem:ODE}
(which 
holds for the argument $\tau-
(\tempo-T_\infty)\in[T_\infty,T_\infty+\Theta(h_\infty))$) yields
\begin{equation}\label{eq:previous}
\int_0^{\ell(\gamma_\varepsilon(\tau))}\kappa_{\gamma_\varepsilon(\tau)}^2\,ds
\le g\big(\tau-(\tempo-T_\infty)\big)\le 2g(T_\infty)=2h_\infty
\qquad
\forall \varepsilon\in(0,\overline\varepsilon], \ \ 
\forall \tau\in
[\tempo,\;\tempo+\Theta(h_\infty)).
\end{equation}

Hence, according to \eqref{eq:useful_inclusion},
inequality \eqref{eq:previous}
holds also on the time interval 
$[\tempo,\,\tempo+\theta/2]$. 
 Arguing as in the proof of Proposition \ref{prop:uniform_estimate_of_higher_derivatives_square}, 
we deduce uniform bounds for all $\Vert \partial^j_s\kappa_{\gae(\tau)}\Vert_{L^2}$, for 
every \( j \in \mathbb{N} \), in the same time interval. 
This yields \eqref{eq:our_aim_is},
and this gives 
the contradiction. When 
$\inidat \in C^\infty([0, \ell(\inidat)])$ and
 all derivatives of $\inidat$ of even order 
at $0$ and $\ell(\inidat)$ vanish, all previous estimates
are valid including the initial time $t=0$,
and
the convergence
takes place 
in 
$C_{\mathrm{loc}}^\infty([0,\Tsingmcf)\times[0,1])$.
\qed

%%%%%%%%%%%%%%%%%%%%%%%%%%%%%%%%%%%%%%%%%%%%%%%%%%%%%%%%%%%%%%%%%%%%%%%%%%%%%
\paragraph*{Acknowledgements.} We wish to thank Alessandra Pluda for stimulating discussions.
All authors are members of the Gruppo Nazionale 
per l’Analisi Matematica, la Probabilità e le loro 
Applicazioni (GNAMPA) of the Istituto Nazionale di Alta Matematica (INdAM).
The work of M.N. was partially supported by Next Generation EU, 
PRIN 2022E9CF89.
R.S. joins the project 2025 CUP E5324001950001. We also acknowledge the partial financial support of PRIN 2022PJ9EFL ``Geometric Measure Theory: Structure of Singular
Measures, Regularity Theory and Applications in the Calculus of Variations”. The latter has
been funded by the European Union under NextGenerationEU. Views and opinions expressed
are however those of the author(s) only and do not necessarily reflect those of the European
Union or The European Research Executive Agency. Neither the European Union nor the
granting authority can be held responsible for them.
%%%%%%%%%%%%%%%%%%%%%%%%%%%%%%%%%%%%%%%%%%%%%%%%%%%%%%%%%%%%%%%%%%%%%%%%%%%%%%%%%%%%%%%%%%%%%%%%%%%%%%%%%%%%%%%%
\bibliographystyle{plain}
\bibliography{refs_dirichlet_rivisto}
%%%%%%%%%%%%%%%%%%%%%%%%%%%%%%%%%%%%%%%%%%%%%%%%%%%%%%%%%%%%%%%%%%%%%%%%%%%%
%%%%%%%%%%%%%%%%%%%%%%%%%%%%%%%%%%%%%%%%%%%%%%%%%%%%%%%%%%%%%%%%%%%%%%%%%%%%
\end{document}